\documentclass[10pt,english,letter, reqno]{amsart}

\usepackage{amssymb,url,xspace,amsthm,adjustbox}

\usepackage{datetime}
\usepackage[T1]{fontenc}
\usepackage{chngcntr}
\usepackage{graphicx}
\usepackage{mathtools}
\usepackage{color}
\usepackage{amsmath}
\usepackage{setspace}
\usepackage{mathrsfs,stmaryrd}
\usepackage{yfonts, bbm}
\usepackage{enumitem}

\usepackage{hyperref}
\hypersetup{
  colorlinks   = true, 
  urlcolor     = blue, 
  linkcolor    = blue, 
  citecolor   = green 
}

\usepackage{todonotes}

\renewcommand{\o}[1]{\buildrel _{\circ} \over {#1}}

\newcommand{\Frob}{{\operatorname{Fr}}}
\DeclareMathOperator{\ord}{ord}

\newcommand{\abs}[1]{{\vert #1 \vert}}

\makeatletter
\newcommand{\labitem}[2]{
\def\@itemlabel{\textbf{#1}}
\item
\def\@currentlabel{#1}\label{#2}}
\makeatother

\newcommand{\Fr}{{\rm {Fr}}}

\newcommand{\bpm}{\begin{pmatrix}}
\newcommand{\epm}{\end{pmatrix}}

\usepackage{float}

\theoremstyle{plain}
      \newtheorem{theorem}{Theorem}[section]
      \newtheorem{proposition}[theorem]{Proposition}
      
      \newtheorem{lemma}[theorem]{Lemma}
      
      \newtheorem{conjecture}[theorem]{Conjecture}

      \theoremstyle{definition}
      \newtheorem{definition}{Definition}

      \newtheorem{remark}[theorem]{Remark}
      \newtheorem{example}[theorem]{Example}

      \newtheorem*{conjecture*}{Conjecture}

\author[A. Fiori]{Andrew Fiori}
\address{Department of Mathematics and Statistics, University of Lethbridge,
4401 University Drive,
Lethbridge, Alberta,
T1K 3M4,
Canada}
\email{andrew.fiori@uleth.ca}
\thanks{Andrew Fiori thanks and acknowledges the University of Lethbridge for their financial support as well as the support of NSERC Discovery Grant RGPIN-2020-05316.}

\author[H. Gheisari. ]{Hiva Gheisari}
\address{Department of Mathematics and Statistics, University of Lethbridge,
4401 University Drive,
Lethbridge, Alberta,
T1K 3M4,
Canada}
\email{hiva.gheisari@uleth.ca}
\thanks{Hiva Gheisari thanks and acknowledges the University of Lethbridge for their financial support.}

\title[Counting Frobenius Pseudoprimes]{Counting Frobenius Pseudoprimes}

\date{\today}                                           

\begin{document}

\begin{abstract}
We generalize the work of Erdos-Pomerance and Fiori-Shallue on counting Frobenius pseudoprimes from the cases of degree one and two respectively to arbitrary degree.
More specifically we provide formulas for counting the number of false witnesses for a number $n$ with respect to Grantham's Frobenius primality test.
We also provide conditional assymptotic lower bounds on the average number of Frobenius pseudoprimes and assymptotic upper bounds on the same.
\end{abstract}

\maketitle

\section{Introduction}

Primality tests are an important tool in many computational problems. 
The fastest tests are non-deterministic and risk false positives.
In order to understand the time and accuracy trade-offs of these primality tests, one must assess the risk of false positives 

The terminology one often uses for false positives in this context is pseudoprimes. If we ensure that a number is composite and satisfies specific properties of prime numbers, we call it a pseudoprime number.

In \cite{Grantham} a common refinement of many of the primality tests. One of the tests that Grantham generalized is the Fermat test, which tries to identify if a number $n$ is prime by checking if $a^n \equiv 1 \pmod{n}$ for some pre-chosen $a$.
In \cite{ErdosPomerance} there is a study of the number of Fermat pseudoprimes and provided upper and lower bounds for the number of Fermat pseudoprimes which are less than $x$ by 
\[x^{15 / 23} \leq \frac{1}{x}\sum_{n \leq x} \left| \left\{ a \in \left(\tfrac{\mathbb{Z}}{n\mathbb{Z}}\right)^{\times} : a^{n-1} \equiv 1 \pmod n \right\} \right| \leq x\mathcal{L}^{-1+o(1)}(x),\]
where $\mathcal{L}(x)= \exp \left(\frac{\log x \log \log\log x}{\log \log x}\right)$, and the sum is only over composite numbers $n$. 
Roughly speaking, for $a$ chosen at random the probability that a composite $n$ is misidentified by the base $a$ test is between $x^{-8/23}$ and $\mathcal{L}(x)^{-1}$.

Rather than a base $a$ Grantham's test uses an auxilliary parameter of a degree $d$ polynomial. In the Fermat case the base $a$ becomes the polynomial $x-a$. 
The work in \cite{FioriShallue} generalizes {Erd}\H{o}s and Pomerance's work to the case of degree two Frobenius pseudoprimes (or equivalently Lucas pseudoprimes), in which they counted the pairs $(P(x),n)$, where $P(x)$ is a quadratic polynomial with integer coefficients and $n$ is Frobenius pseudoprime with respect to $P(x)$. We can call $P(x)$ a liar to integer $n$. They denoted the number of quadratic liars respect with $n$ by $L_2(n)$, and found the following upper and lower bounds for the number of quadratic liars which is analogous to {Erd}\H{o}s and Pomerance in the previous sentences.
\[ x^{1-\alpha^{-1}-o(1)} \leq \frac{1}{x^2}\sum_{n \leq x} L_{2}(n) \leq x \mathcal{L}^{-1+o(1)}(x).\]

This work aims to generalize the above to arbitrary degree Frobenius pseudoprimes that is, where the polynomial $P(x)$ is of higher degree.
So if we denote the number of degree $d$ polynomials which are liars to $n$ by $L_d(n)$ then we aim to establish bounds of the shape
\[ x^{1-\alpha^{-1}-o(1)} \leq \frac{1}{x^d}\sum_{n \leq x} L_{d}(n) \leq x \mathcal{L}^{-1+o(1)}(x).\]
This work is based on, but in fact generalizes, the M.Sc. Thesis, \cite{Hiva} of the second author, which focused on the case of $d=3$.

This work fits into what is a growing area of work attempting to count or tabulate different types of pseudoprimes or otherwise understand the efficacy of the associated test. We point to some pertinent examples here:
\cite{GRANTHAM20101117},
\cite{Shallue2018FastTO},
\cite{Grantham2},
\cite{stephan2020millionsperrinpseudoprimesincluding},
\cite{BFW},
\cite{pomerance2021thoughtspseudoprimes},
\cite{Legnongo2023BadWF},
and
\cite{helmreich2023tabulatingabsolutelucaspseudoprimes}.

The paper is organized as follows:
In Section \ref{sec:background} we introduce Grantham's notion of Frobenius pseudoprimes. At the same time we provide a useful reinterpretation (see Theorem \ref{thm:perm}) which is more amenable to counting.
In Section \ref{sec:count} we provide all the tools necessary to compute
\[  L_d(n) \]
for an arbitrary pair $d$ and $n$. Indeed, we provide the tools to compute more refined counts based on a refined classification of Frobenius pseudoprimes.
In Section \ref{sec:lower} we establish conditional lower bounds on 
\[ \sum_{n<x} L_d(n) .\]
Finally, in Section \ref{sec:upper} we will establish upper bounds on
\[ \sum_{n<x} L_d(n) .\]

\section{Grantham's Definition and Reinterpretations}\label{sec:background}

We begin by recalling some key definitions which we need to give Granthams definition of Frobenius pseudoprimes, see \cite{Grantham}. We will then proceed to overview both Grantham's definition of Frobenius pseudoprimes at the same time as providing some alternative formulations which are amenable to counting. The alternative formulations generalize what was done in \cite{FioriShallue}.

\begin{definition}
Let $f(x), g_1(x), g_2(x)$ be monic polynomials over a commutative ring (with identity). We say that $f(x)$ is the greatest common monic divisor (gcmd) of $g_1(x)$ and $g_2(x)$ if $f(x)$ is monic and the ideal generated by $g_1(x)$ and $g_2(x)$ is equal to the ideal generated by $f(x)$.

We will write $\operatorname{gcmd}_n(g_1(x), g_2(x))$ to denote that the gcmd is being computed in $\mathbb{Z}/n\mathbb{Z}[x]$.

Note that $\operatorname{gcmd}_n(g_1(x), g_2(x))$ does not necessarily exist when $n$ is composite but does exist when $n$ is prime.
\end{definition}

\begin{proposition}
Let $g_1(x), g_2(x)$ be monic polynomials in $\mathbb{Z}[x]$. Then $$f(x)=\operatorname{gcmd}_n\left(g_1(x), g_2(x)\right) \pmod{n},$$ 
if and only if for all $p^r || n$ 
$$f(x) = \operatorname{gcmd}_{p^r}\left(g_1(x), g_2(x)\right) \pmod{p^r},$$
 and the $f(x) \pmod{p^r}$ are all monic and have the same degree.
\end{proposition}

\begin{definition}
Let $f(x)$ be a polynomial in $\mathbb{Z}[x]$.
We shall denote by $\Delta_f$ the polynomial discriminant of $f(x)$.

Recall that $p \nmid \Delta_f$ if and only if $f(x) \pmod{p}$ has no repeat roots in $\overline{\mathbb{F}}_p$.
\end{definition}

\begin{definition}
Let $f(x)$ be a monic polynomial in $\mathbb{Z}[x]$.
By a root of $f(x)$ modulo $p^r$ we shall mean an element of
$\alpha \in \overline{\mathbb{Z}_p}^{\rm ur}\!\!/p^r\overline{\mathbb{Z}_p}^{\rm ur}$, where $\overline{\mathbb{Z}_p}^{\rm ur}$ denotes the ring of integers in the maximal unramified extension of $\mathbb{Q}_p$, such that $f(\alpha) = 0 \pmod{p^r}$.

Recall that for a polynomial $f$ of degree $d$ that if $p\nmid \Delta_f$ then there are $d$ distinct roots modulo $p^r$.
\end{definition}

\begin{proposition}
Let $g_1(x), g_2(x)$ be monic polynomials in $\mathbb{Z}[x]$.
Suppose that $p \nmid \Delta_{g_1}\Delta_{g_2}$.
Write $f_p(x) = {\rm gcmd}_p(g_1(x),g_2(x))$.
Then ${\rm gcmd}_{p^r}(g_1(x),g_2(x))$ exists if and only if for each root $\alpha\in \overline{\mathbb{F}_p}$ of $f_p(x)$ the unique Hensel lifts $\tilde{\alpha}$ of $\alpha$ as a root of $g_1(x)$ and $g_2(x)$ are equal modulo $p^r$.

Moreover, the extended Euclidean algorithm performed modulo $p^r$ will compute it when it exists.
\end{proposition}
\begin{proof}
The forward direction is obvious. 

For the reverse direction it is clear that $f_{p^r}(x) = \prod_{\alpha}(x-\tilde{\alpha}) \in \mathbb{Z}/p^r\mathbb{Z}[x]$ will divide both $g_1(x)$ and $g_2(x)$ modulo $p^r$ as the $\tilde{\alpha}$ are roots of both.
We construct $\tilde{g_1}(x)$ and $\tilde{g_2}(x)$ where $g_i(x) = \tilde{g_i}(x)\pmod{p^r}$ and the Hensel lifts of all roots $\alpha$ agree all the way to $\overline{\mathbb{Z}_p}$. 

Now we may use the extended Euclidean algorithm to find $f(x) = {\rm gcd}(\tilde{g_1}(x),\tilde{g_2}(x))$ over $\mathbb{Z}_p$. It follows that the extended Euclidean algorithm will find \[ f_{p^r}(x) = {\rm gcmd}_{p^r}(\tilde{g_1}(x),\tilde{g_2}(x)) = {\rm gcmd}_{p^r}({g_1}(x),{g_2}(x)).\qedhere\]
\end{proof}

In order to help us study Grantham's definition we will define two different versions of several steps. We begin with his Factorization step.

\begin{definition}[$\text{Factorization}^{\prime}$ step]
Let $f(x) \in \mathbb{Z}[x]$ be a monic polynomial of degree $d$ with discriminant $\Delta$ and consider the 
odd integer $n>1$. Assume that $(n, f(0) \Delta)=1$

We recursively define:
\begin{itemize}
    \item Let $f_0(x)= f(x)$.
    \item For $1 \leq i$, let $F_i(x)= \operatorname{gcmd}_n\!(x^{n^i}-x, f_{i-1}(x))$ and $f_i(x)=\frac{f_{i-1}(x)}{F_i(x)}$
    provided the gcmd can be computed with the Euclidean algorithm.
    \item If at the $j$th step the gcmd cannot be computed using the Euclidean algorithm then $F_i(x)$ are undefined for $i\ge j$.
\end{itemize}

We say a number $n$ passes the $\text{factorization}^{\prime}$ step if $F_i(x)$ are defined for $i>0$ and $f_{d}(x)=1$.

In this case $f_{j}(x)=1$ for $j\ge d$ and we will have
\[ f(x) = F_1(x)F_2(x) \cdots F_d(x) \pmod{n}. \]
\end{definition}
\begin{remark}
If $n$ fails the factorization step then $n$ is composite..
\end{remark}

\begin{definition}[Factorization Step]
Let $f(x) \in \mathbb{Z}[x]$ be a monic polynomial of degree $d$ with discriminant $\Delta$ and consider the 
odd integer $n>1$. Assume that $(n, f(0) \Delta)=1$

We recursively define:
\begin{itemize}
    \item Let $f_0(x)= f(x)$.
    \item For $1 \leq i$, let $F_i(x)= \operatorname{gcmd}_n\!(x^{n^i}-x, f_{i-1}(x))$ and $f_i(x)=\frac{f_{i-1}(x)}{F_i(x)}$
    provided the gcmd exists.
    \item If at the $j$th step the gcmd does not exist then $F_i(x)$ are undefined for $i\ge j$.
\end{itemize}

We say a number $n$ passes the factorization step if $F_i(x)$ are defined for $i>0$ and $f_{d}(x)=1$.
As before, in this case $f_{j}(x)=1$ for $j\ge d$ and we will have
\[ f(x) = F_1(x)F_2(x) \cdots F_d(x) \pmod{n}. \]
\end{definition}
\begin{remark}
    The difference between the above two definitions is that in the second we do not insist on using the Euclidean algorithm to find the gcmd.

    This leads to a strictly weaker primality test, however, results in a definition which is more amenable to counting.
\end{remark}

\begin{definition}
Suppose $f(x) \in \mathbb{Z}[x]$ is degree $d$.
    If $(f,n)$ passes Factorization step we shall refer to the sequence of numbers 
   \[ [\sigma] = \left(\frac{{\rm deg} F_1(x)}{1}, \frac{{\rm deg} F_2(x)}{2},\ldots, \frac{{\rm deg} F_{d}(x)}{d}\right), \]
    as the cycle structure of the pair $(f,n)$.
The sequence of numbers do determine a conjugacy class in $S_d$, this terminology will be further justified by Theorem \ref{thm:perm}.
    \end{definition}

\begin{definition}[Factorization Step Modulo $p^r$]
Let $f(x) \in \mathbb{Z}[x]$ be a monic polynomial of degree $d$ with discriminant $\Delta$ and consider the 
odd integer $n>1$. Assume that $(n, f(0) \Delta)=1$.
Let $p^r|n$.

We recursively define:
\begin{itemize}
    \item Let $f_0(x)= f(x)$.
    \item For $1 \leq i$, let $F_i(x)= \operatorname{gcmd}_{p^r}\!(x^{n^i}-x, f_{i-1}(x))$ and $f_i(x)=\frac{f_{i-1}(x)}{F_i(x)}$
    provided the gcmd exists.
    \item If at the $j$th step the gcmd does not exist then $F_i(x)$ are undefined for $i\ge j$.
\end{itemize}
We say a number $n$ passes the factorization step modulo $p^r$ if $F_i(x)$ are defined for $i>0$ and $f_{d}(x)=1$.
As before if $f_{j}(x)=1$ for $j\ge d$ and we will have
\[ f(x) = F_1(x)F_2(x) \cdots F_d(x) \pmod{p^r} \] 
and we shall refer to the sequence of numbers $\frac{{\rm deg} F_i(x)}{i}$ as the cycle structure of the pair $(f,n)$ modulo $p^r$.
\end{definition}

\begin{proposition}\label{prop:fs-fsp}
A pair $(f,n)$ passes the factorization step if and only if for all $p^r||n$ the pair $(f,n)$ passes the factorization step modulo $p^r$ and for each $p^r$ the cycle structures of $(f,n)$ modulo $p^r$ are the same.

In this case the cycle structures modulo $p^r$ will be the same as the cycle structure modulo $n$.
\end{proposition}
\begin{proof}
This is essentially immediate using the Chinese remainder theorem, see also \cite[Lemma 3.25, Proposition 3.26, Theorem 3.28]{Hiva} for a detailed argument.
\end{proof}

\begin{definition}[Frobenius Step]
Suppose $(f,n)$ passes the factorization step.

We say that $(f,n)$ passes the Frobenius Step if
for $2 \leq i \leq d$, we have
\[ F_i\left(x^n\right) = 0 \pmod {n, F_i(x)}. \]

If $(f,n)$ fails the Frobenius step then $n$ is composite.
\end{definition}

\begin{definition}[Frobenius Step modulo $p^r$]
Suppose $(f,n)$ passes the factorization step modulo $p^r$.

We say that $(f,n)$ passes the Frobenius Step if
for $2 \leq i \leq d$, we have
\[ F_i\left(x^n\right) = 0 \pmod {p^r, F_i(x)}. \]
\end{definition}

\begin{proposition}\label{prop:frs-frsp}
Suppose $(f,n)$ passes the factorization step.

The pair $(f,n)$ passes the Frobenius step if and only if for each $p^r ||n$ we have that $(f,n)$ passes the Frobenius step modulo $p^r$.
\end{proposition}
\begin{proof} 
As with Proposition \ref{prop:fs-fsp} this is essentially immediate using the Chinese remainder theorem, see also \cite[Lemma 3.25, Proposition 3.26, Theorem 3.28]{Hiva} for a detailed argument.
\end{proof}

\begin{definition}
Let $f(x) \in \mathbb{Z}[x]$ be a monic polynomial. We say that $n$ is a  Frobenius pseudoprime modulo $p^r$ with respect to $f(x)$ if it passes both the factorization step modulo $p^r$ and the Frobenius step modulo $p^r$.
\end{definition}

\begin{definition}[Jacobi Step]
Suppose $(f,n)$ passes Factorization step.
Let
\[ S=\sum_{2 \mid i} \frac{\deg \left(F_i(x)\right)}{i}. \]
We say that $(f,n)$ passes the Jacobi step if 
\[ (-1)^S = \left(\frac{\Delta_f}{n}\right). \] 
If $(f,n)$ fails the Jacobi step the $n$ is composite.

Notice that $(-1)^S = {\rm sgn}(\sigma)$ where $\sigma$ is the cycle structure of $(f,n)$.
\end{definition}

A slightly improved version of the Jacobi step is the following
\begin{definition}[$\text{Jacobi}^{\prime}$ Step]

Suppose $(f,n)$ passes Factorization step. Denote by $\Delta_{F_i}$ the discriminant of $F_i$.

We say that $(f,n)$ passes the $\text{Jacobi}^{\prime}$ step if 
\[ (-1)^{(i-1)\frac{\deg \left(F_i(x)\right)}{i}} = \left(\frac{\Delta_{F_i}}{n}\right). \]
If $(f,n)$ fails the $\text{Jacobi}^{\prime}$ step the $n$ is composite.

Notice that $(-1)^{(i-1)\frac{\deg \left(F_i(x)\right)}{i}} = {\rm sgn}(\sigma)$ where $[\sigma]$ is the cycle structure of $(F_i,n)$.
\end{definition}
\begin{remark}
Passing the $\text{Jacobi}^{\prime}$ step implies passing the Jacobi step.

Passing the Jacobi step implies passing the $\text{Jacobi}^{\prime}$ step for degree $1$, $2$, and $3$,
and in degree $4$ unless the cycle structure is $(2,1,0,0)$, ie $F_1$ and $F_2$ are degree $2$.

\end{remark}

\begin{definition}\label{def:pspn}
Let $f(x) \in \mathbb{Z}[x]$. We say that $n$ is a Frobenius pseudoprime with respect to a monic polynomial $f(x)$ if
it passes the $\text{Factorization}^{\prime}$ step, the Frobenius step, and the Jacobi step.

If we wish to substitute the $\text{Jacobi}^{\prime}$ step for the Jacobi step we will say it is a Frobenius pseudoprime using the $\text{Jacobi}^{\prime}$ step.

Many different terminologies exist, one may for example call $f$ a liar for $n$ or a false witness.
\end{definition}

\begin{proposition}\label{prop:pspn-pspr}
The pair $(f,n)$ is a Frobenius pseudoprime pair if and only if for all $p^r||n$ we have that $(f,n)$ is a Frobenius pseudoprime modulo $p^r$, all of the cycle structures modulo $p^r$ are the same, and $(f,n)$ passes the Jacobi step.
\end{proposition}
\begin{proof} 
This follows immediately from Propositions \ref{prop:fs-fsp} and \ref{prop:frs-frsp} but see also \cite[Lemma 3.25, Proposition 3.26, Theorem 3.28]{Hiva} for a detailed argument. 
\end{proof}

\begin{remark}
Note that in Granthams definition he used the factorization step, Frobenius Step and Jacobi step whereas we are using the $\text{factorization}^{\prime}$ step, Frobenius Step and either the Jacobi or $\text{Jacobi}^{\prime}$ step.
\end{remark}

\begin{definition}
Fix $n$ and $p^r | n$, we define $\phi_n$ to be the map 
\[ \phi_n : \left(\overline{\mathbb{Z}_p}^{\rm ur}/(p^r)\right)^\times \rightarrow \left(\overline{\mathbb{Z}_p}^{\rm ur}\!\!/p^r\overline{\mathbb{Z}_p}\right)^\times \]
given by $x\mapsto x^n$.

Recalling the identification ${\rm Gal}(\overline{\mathbb{Q}_p}^{\rm ur}/\mathbb{Q}_p) \sim {\rm Gal}(\overline{\mathbb{F}_p}/\mathbb{F}_p)$
we shall denote by 
\[ {\rm Fr}_p: \left(\overline{\mathbb{Z}_p}^{\rm ur}\!\!/p^r\overline{\mathbb{Z}_p}^{\rm ur}\right)^\times \rightarrow \left(\overline{\mathbb{Z}_p}^{\rm ur}/(p^r)\right)^\times\] the Frobenius map.

We recall that a polynomial $f(x)\in \overline{\mathbb{Z}_p}/(p^r)[x]$ with $p\nmid \Delta_f$ is in $\mathbb{Z}/(p^r)[x]$ if and only if ${\rm Fr}_p$ permutes the roots of $f$ modulo $p^r$. When $f(x) \in \mathbb{Z}/(p^r)[x]$ we shall refer to the cycle type of ${\rm Fr}_p$ as the factorization type of $f(x)$.

Finally, we note that $\phi_n$ and ${\rm Fr}_p$ commute, that is $\phi_n\circ {\rm Fr}_p = {\rm Fr}_p \circ \phi_n$.
\end{definition}

\begin{remark}
We recall that conjugacy classes in $S_d$ classify elements in $S_d$ up to relabeling of elements in the set $\{ 1,\ldots, d\}$ on which they are acting.
We recall that conjugacy classes of elements in $S_d$ are classified by cycle types. 

 We shall eventually be considering simultaneous conjugacy classes of a pair of commuting elements $\sigma,\tau\in S_d$. The classification of such pairs is a combinatorial exercise, but, is more involved than simply considering two cycle types.
\end{remark}

\begin{proposition}\label{prop:permuteiff}
Fix $n$ and $p^r | n$
    Suppose $f \in \mathbb{Z}/p^r\mathbb{Z}[x]$ with $p \nmid \Delta_ff(0)$ then
    the map $\phi_n: x\mapsto x^n$ induces a permutation of the roots of $f$ modulo $p^r$ if and only if
    \begin{enumerate}
        \item $f(x^n) = 0 \pmod{f(x),p^r}$ 
        \item there exists $i\in \mathbb{N}^+$ with $x^{n^i}-x = 0 \pmod{f(x),p^r}$.
    \end{enumerate}
\end{proposition}
\begin{proof}
Let $M = \{ \alpha\in \overline{\mathbb{Z}_p}^{\rm ur}\!\!/p^r\overline{\mathbb{Z}_p}^{\rm ur}\,\mid\, f(\alpha) = 0\}$, that is the roots of $f$ modulo $p^r$.

Condition (1) is that $f(x) | f(x^n)$ modulo $p^r$. Since $p \nmid \Delta_f$ we know that $f$ has no repeat roots and hence the divisibility condition is equivalent to the condition that root of $f$ modulo $p^r$ is a root of $f(x^n)$, that is $\forall \alpha\in M, f(\phi_n(\alpha)) = 0$. Or equivalently, 
$\phi_n|_M: M \rightarrow M$.

Thus $\phi_n$ inducing a permutation of $M$ certainly implies (1) and conversely condition (1) implies at least that $\phi_n$ incudes a map on $M$, but not necessarily bijective.

Now, since $M$ is finite, $\phi_n : M\rightarrow M$ is bijective if and only if there exists $i\in \mathbb{N}^+$ with $\phi^i_n|_M = {\rm Id}_M$. Condition (2) is easily seen to be equivalent to this.
\end{proof}

\begin{proposition}\label{prop:cycle}
Fix $n$ and $p^r | n$.
Suppose $f \in \mathbb{Z}/p^r\mathbb{Z}[x]$ with $p \nmid \Delta_ff(0)$ then $(f,n)$ is a Frobenius pseudoprime modulo $p^r$ if and only if  $\phi_n : x\mapsto x^n$ induces a permutation of the roots of $f$ modulo $p^r$
\end{proposition}
\begin{proof}
If $(f,n) $ is a Frobenius pseudoprime modulo $p^r$ then we may write 
\[ f(x) = F_1(x)\cdots F_d(x)\] where each of $F_i$ satisfy $F_i(x) | F_i(x^n)$ and $F_i(x) | x^{n^i}-x$. It follows from Proposition \ref{prop:permuteiff} that $\phi_n$ permutes the roots of each $F_i$ modulo $p^r$ and hence of $f$ modulo $p^r$.

Conversely, suppose $\phi_n$ permutes the roots of $f$ modulo $p^r$.
As above let $M$ denote the roots of $f$ modulo $p^r$. Now define
\[ M_i = \{ \alpha\in M\,|\, \phi_n^i(\alpha) = \alpha,\,\forall 0<j<i, \phi_n^j(\alpha)\neq \alpha \} \]
and then
\[ G_i(x) = \prod_{\alpha\in M_i} x-\alpha. \]
Because $\Frob_p$ and $\phi_n$ commute we have that ${\rm Fr}_p$ will act on the sets $M_i$ and hence $G_i(x)\in \mathbb{Z}/p^r\mathbb{Z}[x]$.
Moreover, by construction we will have 
\[ f(x) = G_1(x)\cdots G_d(x). \]
Additionally we have by construction that $G_i(x) | G_i(x^n)$, that $G_i(x) | x^{n^i}-x$, and that $G_i(x) \nmid x^{n^j}-x$ for $j<i$.
From these divisibility conditions we find that in the Factorization step modulo $p^r$ we will have $F_i(x) = G_i(x)$ so that this step will be passed.
The Frobenius step will be passed as $G_i(x) | G_i(x^n)$.
\end{proof}

\begin{theorem}\label{thm:perm}
Let $f$ be a degree $d$ polynomial. Suppose ${\rm gcd}(n,\Delta_ff(0))=1$.
The pair $(f,n)$ is a Frobenius pseudoprime (using the Jacobi step) with cycle structure $\sigma \in S_d$ if and only if for all prime $p^r || n$ we have
    \begin{enumerate}
        \item The map $\phi_n$ permutes the roots of $f(x)$ modulo $p^r$ with cycle structure  $\sigma$,
        \item $\left(\frac{\Delta}{n}\right)={\rm sgn}(\sigma)$.
    \end{enumerate}
\end{theorem}
\begin{proof} This is immediate from Propositions \ref{prop:pspn-pspr} and \ref{prop:cycle}.
See also \cite[Theorem 3.28]{Hiva} for a detailed argument.
\end{proof}

\begin{remark}\label{rem:combining}
A pair $(f,n)$ passing the factorization step is a Frobenius pseudoprime under the $\text{Jacobi}^{\prime}$ step if and only if each pair $(F_i,n)$ is a Frobenius pseudoprime.
For this reason, while considering the Frobenius test it is most interesting to consider the cases where $f=F_i$ for some $i$.

More generally, if $(f,n)$ and $(g,n)$ are Frobenius pseudoprimes with cycle structures $\sigma_f$ and $\sigma_g$ with ${\rm gcd}(f,g)=1$ then $(fg,n)$ is also a Frobenius pseudoprime where the cycle structure is given by $[\sigma_f]+[\sigma_g]$. That is, the number of $i$-cycles is the total between that in $\sigma_f$ and $\sigma_g$.

We can also understand this cycle structure as follows:
There exists a natural map $S_{d_1}\times S_{d_2}\rightarrow S_{d_1+d_2}$, or, more generally from  $S_{d_1}\times \cdots \times S_{d_k} \rightarrow S_{d_1+\cdots+d_k}$. If we denote the image of $(\sigma_f,\sigma_g)$ under this map by $\sigma_f \times \sigma_g$ then the cycle structure of $fg$ is that of $\sigma_f \times \sigma_g$.

If follows that if the cycle type $\sigma$ can be decomposed then the Frobenius test is not stronger than performing multiple Frobenius tests with lower degree polynomials. Conversely, if the cycle type of $\sigma$ has only a single cycle then $f$ cannot come from Frobenius pseudoprimes of lower degree.
For the reason the situation $f=F_d$ where $d=\deg(f)$ are the most interesting to consider.
\end{remark}

\section{Computing $L_d(n)$}\label{sec:count}

\begin{definition}
We denote by $L_d(n)$ the number of polynomials $f \pmod{n}$ of degree $d$ such that $(f,n)$ is a Frobenius pseudoprime.

Let $\sigma$ be in ${\rm S_d}$ we denote by $L_d^{\sigma}(n)$ the number of $f \pmod{n}$ of degree $d$ such that $(f,n)$ is a Frobenius pseudoprime with cycle structure $\sigma$.
\end{definition}

\begin{proposition}
If $\sigma_1\in S_{d_1}$ and $\sigma_2\in S_{d_2}$ have no cycle lengths in common then 
\[ L_{d_1+d_2}^{\sigma_1\times\sigma_2}(n) \ge L_{d_1}^{\sigma_1}(n)L_{d_2}^{\sigma_2}(n). \]
If additionally $\sigma_1$ and $\sigma_2$ have no cycle lengths in common, and we have used the $\text{Jacobi} ^{\prime}$ condition, then
\[ L_{d_1+d_2}^{\sigma_1\times\sigma_2}(n) = L_{d_1}^{\sigma_1}(n)L_{d_2}^{\sigma_2}(n). \]
We recall the notation $\sigma_1\times\sigma_2$ is defined in Remark \ref{rem:combining}.
\end{proposition}
\begin{proof}
This is immediate from Remark \ref{rem:combining}.
\end{proof}
\begin{remark}
From the above, the most important cases to count are those where $\sigma$ is a product of disjoint cycles all having the same lengths
 and likely the best tests would come from having $\sigma$ be a single $d$-cycle.
\end{remark}

\begin{definition}
We denote by $L_d^{\sigma}(n,p^r)$ the number of polynomials $f(x) \in \mathbb{Z}/p^r\mathbb{Z}[x]$ such that $(n,f(x))$ is a Frobenius pseudoprime modulo $p^r$ with cycle structure $\sigma$.

We denote by $L_d^{\sigma,+}(n,p^r)$ (respectively $L_d^{\sigma,-}(n,p^r)$) the number of polynomials $f(x) \in \mathbb{Z}/p^r\mathbb{Z}[x]$ such that $(n,f(x))$ is a Frobenius pseudoprime modulo $p^r$ with cycle structure $\sigma$ where $(\Delta_f,p) = +1$ (respectively $-1$).

Suppose $\sigma,\tau\in {\rm S}_d$ are representatives for an equivalence class of commuting elements.
We denote by 
$L_d^{\sigma,\tau}(n,p^r)$ the number of polynomials $f(x) \in \mathbb{Z}/p^r\mathbb{Z}[x]$ such that $(n,f(x))$ is a Frobenius pseudoprime modulo $p^r$ with cycle structure $\sigma$, where $\Frob_p$ has factorization structure $\tau$ and the pair of permutations $[(\sigma,\tau)] \sim [(\phi_n, \Frob_p)]$, that is, under an identification of ${\rm S}_d$ with the symmetric group on the roots of $f(x)$ modulo $p^r$ the pairs would be simultaneously conjugate. Ie: there exists $g\in {\rm S}_d$ with $g\sigma g^{-1} = \phi_n$ and $g\tau g^{-1} = \Frob_p$.
\end{definition}

Going forward we shall need to make use of several notions related to group actions. We shall always be consider
\begin{definition}
An action of $G$ on a set $X$ means a map from 
\[ G \rightarrow {\rm Aut}(X). \]
This map need not be injective. The standard example of such is that any subgroup $G\in {\rm S}_d$ acts on $\{1,\ldots, d\}$.

By a sub-action of $X$ we mean a subset $S \subseteq X$ such that $\forall g\in G, g(S) = S$.

By an irreducible action we mean an action with no proper sub-actions.
We note that this equivalent to the action being transitive.

By an isomorphism of two actions, say with $G$ acting $X$ and $G$ acting on $Y$, we mean a bijection
$f: X \rightarrow Y$ such that $\forall g\in G, g\cdot f(x) = f(g\cdot x)$. We shall write $X\sim Y$ to indicate two actions are isomorphic.

Throughout this paper the group which is acting is typically $\mathbb{Z}\times \mathbb{Z}$ presented as being generated by commuting elements $\sigma$ and $\tau$. Giving an action of this group on $\{1,\ldots, d\}$ is then equivalent to specifying $\sigma_1,\tau_1\in S_d$ which commute.

In this special case where $\sigma_1,\tau_1\in S_d$ commute and $\sigma_2,\tau_2\in S_d$ also commute then we have that the actions are isomorphic if and only if there exists $g\in S_d$ with $\sigma_1 = g\sigma_2 g^{-1}$ and $\tau_1 = g\tau_2 g^{-1}$. 

Another group which will often act is a group $\langle \phi_n, \Fr \rangle$ which is also abstractly a quotient of $\mathbb{Z}\times \mathbb{Z}$. When we say an action of $\langle \sigma,\tau \rangle$ and $\langle \phi_n, \Fr \rangle$ are isomorphic we mean that the corresponding actions of $\mathbb{Z}\times \mathbb{Z}$ are isomorphic.
\end{definition}

\begin{proposition}
For any pair $\sigma,\tau$ we have
\[ L_{d}^{\sigma,\tau}(n,p^r) = L_{d}^{\sigma,\tau}(n,p) \]
or more precisely, the map from $\mathbb{Z}/p^r\mathbb{Z}[x] \rightarrow \mathbb{Z}/p\mathbb{Z}[x]$, induces a bijection from $f$ for which $(f,n)$ is a Frobenius pseudoprimes modulo $p^r$ with cycle structure $\sigma$ and factorization structure $\tau$ to the same modulo $p$.
\end{proposition}
\begin{proof}
Hensel's lemma allows us to lift each root $\alpha$ of $F_i(x) \pmod{p}$ to a unique root $\tilde{\alpha}$ of $x^{n^i}-x \pmod{p^r}$. 
Notice that since $\phi_n$ permuted roots the roots of $x^{n^i}-x$ modulo $p^r$ and of $F_i(x)$ modulo $p$, the lift of $\phi_n^j(\alpha)$ must be $\phi_n^{j}(\tilde{\alpha})$.
It follows that if we construct a polynomial $\widetilde{F_i}$ modulo $p^r$ whose roots are the lifts of those of $F_i$ then $\phi_n$ will permute the roots of $\widetilde{F_i}$.
Thus this construction allows us to define the inverse to the reduction map.
\end{proof}

\begin{proposition}
For any $\sigma\in S_d$ we have the following formula for $L_d^{\sigma}(n)$ when ${\rm sgn}(\sigma)=+1$ we have:
    \[ L_d^{\sigma}(n)= \frac{1}{2} \left(\prod_{p^r |\!| n}\left( L_d^{\sigma,+}(n,p ) + L_d^{\sigma,-}(n,p )\right) +\prod_{p^r |\!| n} \left( L_d^{\sigma,+}(n,p ) +(-1)^{r} L_d^{\sigma,-}(n,p )\right) \right). \]
Likewise for $L_d^{\sigma}(n)$ when ${\rm sgn}(\sigma)=-1$ we have:
    \[ L_d^{\sigma}(n)= \frac{1}{2} \left(\prod_{p^r |\!| n}\left( L_d^{\sigma,+}(n,p ) + L_d^{\sigma,-}(n,p )\right) -\prod_{p^r |\!| n} \left( L_d^{\sigma,+}(n,p ) +(-1)^{r} L_d^{\sigma,-}(n,p )\right) \right). \]
Where in each case we have
\[  L_d^{\sigma,+}(n,p ) = \sum_{{\rm sgn}(\tau)=1} L_d^{\sigma,\tau}(n,p )\]
and
\[ L_d^{\sigma,-}(n,p ) = \sum_{{\rm sgn}(\tau)=-1} L_d^{\sigma,\tau}(n,p )\]
where the sums are over $\tau\in S_d$ with $\sigma\tau=\tau\sigma$.
\end{proposition}

\begin{proposition}\label{prop:sets}
    For any pair $\sigma,\tau$ which act on $\{1,\ldots, d\}$
    then $L_{d}^{\sigma,\tau}(n,p)$ is equal to the number of subsets $S\subset \overline{\mathbb{F}_p}$ of size $d$ 
    stable under the action of $\phi_n,\Frob_p$, that is $\phi_n(S)=S$ and $\Fr_p(S)=S$, and where their action is isomorphic to the action of $\sigma,\tau$ on $\{1,\ldots, d\}$. That is, the permutation action of $\phi_n,\Frob_p$ is conjugate to the action of $\sigma,\tau$ on $\{1,\ldots, d\}$ under some, equivalently all, 
    identification $\{1,\ldots, d\}$ with $S$.
\end{proposition}
\begin{proof}
This is simply the observation that polynomials are determined by their collection of roots which must be $\Frob_p$ stable in order to be defined over $\overline{\mathbb{F}_p}$ and $\phi_n$ stable in order to correspond to a Frobenius pseudoprime modulo $p^r$.
\end{proof}

\begin{proposition}
    Suppose $\sigma_1,\tau_1\in S_{d_1}$ and  $\sigma_2,\tau_2\in S_{d_2}$ are two pairs of commuting elements for which in the respective actions on $\{1,\ldots, d_1\}$ and $\{1,\ldots, d_2\}$ there are no isomorphic sub-actions. 
    Then
    \[ L_{d_1+d_2}^{\sigma_1\times\sigma_2,\tau_1\times\tau_2}(n,p)= 
 L_{d_1}^{\sigma_1,\tau_1}(n,p) L_{d_2}^{\sigma_2,\tau_2}(n,p). \]
 Finally, if $\sigma=\sigma_1\times\cdots\times \sigma_1$ and $\tau = \tau_1\times\cdots\tau_1$ are each a $k$-fold product of the same irreducible action then
 \[ L_{d}^{\sigma,\tau}(n,p) = \begin{pmatrix} L_{d_1}^{\sigma_1,\tau_1}(n,p) \\ k \end{pmatrix} .\]
 We recall the notation $\sigma_1\times\cdots \times \sigma_1$ and $\tau_1\times\cdots \times \tau_1$ are defined in Remark \ref{rem:combining}.
\end{proposition}
\begin{proof}
By Proposition \ref{prop:sets} the roots of a polynomial corresponding to an element of $ L_{d_1+d_2}^{\sigma_1\times\sigma_2,\tau_1\times\tau_2}(n,p)$ decomposes into a disjoint union of two sets $S_1$ and $S_2$ where the actions are respectively those of $\sigma_1,\tau_1$ and $\sigma_2,\tau_2$, that is, come from polynomials counted by $L_{d_1}^{\sigma_1,\tau_1}(n,p)$ and $L_{d_2}^{\sigma_2,\tau_2}(n,p)$ respectively.

That $\sigma_1$ and $\sigma_2$ have no isomorphic sub-representations implies that roots are automatically distinct.

This is the difference in the case where we consider  $\sigma=\sigma_1\times\cdots\times \sigma_1$ and $\tau = \tau_1\times\cdots\tau_1$.
To construct an element for $L_{d}^{\sigma,\tau}(n,p)$ we must select $k$-distinct polynomials associated to $ L_{d_1}^{\sigma_1,\tau_1}(n,p) $ to ensure the product would have no repeat roots.

This is immediate from Proposition \ref{prop:sets} and the condition that roots must be distinct. 
\end{proof}
\begin{remark}
From the above the most important pairs $\sigma,\tau$ are those for which the action of $\langle \sigma,\tau\rangle$ on $\{1,\ldots, d\}$ is transitive.
This is already guaranteed when $\sigma$ acts as a $d$-cycle.

If $\sigma,\tau\in S_d$ commute and $\sigma$ is a $d$-cycle then $\tau=\sigma^i$ for some $i$.
\end{remark}

\begin{proposition}\label{prop:orbitcount1}
Fix $n$ and $p^r|n$.
Fix natural numbers $d$ and $\ell$.
For each $\delta | d\ell$ define the set $I_\delta$ to be the set of equivalences classes of commuting pairs $(\sigma,\tau)$ each from $S_\delta$ where ${\rm ord}(\sigma)|d$, ${\rm ord}(\tau)|\ell$ and
$\langle \sigma,\tau \rangle$ acts transitively.
Then
\[ {\rm gcd}(n^d-1,p^\ell-1) = \sum_{\delta| d\ell} \delta\sum_{(\sigma,\tau)\in I_\delta}  L_{\delta}^{\sigma,\tau}(n,p) \]
\end{proposition}
\begin{proof}
We have that the group $\langle \phi_n,\Frob_p\rangle$ acts on the set 
\[ \left\{ \alpha \in \overline{\mathbb{F}_p}^\times \;|\; {\rm \ord}(\alpha) | {\rm gcd}(n^{d}-1,p^\ell-1) \right\} \]
which has size $ {\rm gcd}(n^{d}-1,p^\ell-1) $.
We write 
\[ \left\{ \alpha \in \overline{\mathbb{F}_p}^\times \;|\; {\rm \ord}(\alpha) | {\rm gcd}(n^{d}-1,p^\ell-1) \right\}  = \bigsqcup_i F_i \]
as a disjoint union of orbits.
By the orbit stabilizer theorem the size of every orbit $F_i$ will satisfy $|F_i| = \delta$ for some $\delta|d\ell$.
When $|F_i|=\delta$ then by definition of $I_\delta$ the action of $\phi_n,\Fr_p$ on $F_i$ must be equivalent to some representative from $I_\delta$ hence we have
\[ {\rm gcd}(n^{d}-1,p^\ell-1)= \sum_{i} |F_i| = \sum_{\delta|d\ell} \sum_{(\sigma,\tau)\in I_\delta} \sum_{F_i \sim (\sigma,\tau)} |F_i| 
\]
where $F_i \sim (\sigma,\tau)$ indicates those $F_i$ where action of $\phi_n,\Fr_p$ is equivalent to that of $(\sigma,\tau)\in I_\delta$.
Now using that for $(\sigma,\tau)\in I_\delta$ and Proposition \ref{prop:sets} we have
\[ \sum_{F_i \sim (\sigma,\tau)} |F_i| = \delta L_{\delta}^{\sigma,\tau}(n,p).\]
Substituting this into the above gives the result.
\end{proof}

The following proposition is the key to classifying irreducible actions of commuting elements $\sigma,\tau$. We note that when $d$ is small it is simpler to just enumerate them manually.
\begin{proposition}\label{prop:classifyactions}
Suppose $\sigma, \tau \in S_{\delta}$ commute, ${\rm ord}(\sigma)=d$, ${\rm ord}(\tau)=\ell$ and $\langle \sigma,\tau \rangle$ acts irreducibly then there exists $(s,r) \in \mathbb{Z}/d\mathbb{Z}\times\mathbb{Z}/\ell\mathbb{Z}$ of the form $(s,r) = (d/o,a\ell/o)$ with $o|{\rm gcd}(d,\ell)$ and ${\rm gcd}(a,\ell)=1$ such that
\[ \langle \sigma,\tau \rangle = \langle \sigma,\tau\;| \sigma^d = \tau^\ell = \sigma^{s}\tau^{-r} = 1 \rangle. \]
\end{proposition}
\begin{proof}
We are essentially classifying quotients of $\langle \sigma,\tau\;|\; \sigma^d=\tau^\ell=1 \rangle$ where the images of $\sigma$ and $\tau$ continue to have order $d$ and $\ell$ respectively.

By applying the Chinese Remainder Theorem we are reduced to considering 
\[ \{ 0 \} \times \mathbb{Z}/p^v\mathbb{Z}\qquad \mathbb{Z}/p^u\mathbb{Z} \times \{0\}  \qquad \mathbb{Z}/p^u\mathbb{Z} \times \mathbb{Z}/p^v\mathbb{Z}. \]
In either of the first two cases any non-trivial relation implies the order of $\sigma$ or $\tau$ is lower than what is assumed.
In the third case a general relation would be of the form:
\[ (a_1p^{b_1}, a_2p^{b_2}) \]
By multiplying by $p^{{\rm min}(u-b_1,v-b_2)}$ we see that we must have had $t=u-b_1=v-b_2$ to avoid having a relation which lowers the order of $\sigma$ or $\tau$.
Finally, we may multiply by $a_1^{-1}$ to obtain the form 
\[ (p^{u-t},ap^{v-t})\]
with $0\leq t\leq {\rm min}(u,v)$ which gives the result.
\end{proof}

\begin{proposition}\label{prop:orbitcount2}
Fix $n$ and $p^r|n$.
Fix natural numbers $d$, $\ell$, and $(s,r) \in \mathbb{Z}/d\mathbb{Z}\times \mathbb{Z}/\ell\mathbb{Z}$
Consider the set $I_\delta$ of equivalences classes of commuting pairs $(\sigma,\tau)$ from $S_\delta$ where ${\rm ord}(\sigma)|d$, ${\rm ord}(\tau)|\ell$, $\sigma^{s}\tau^{-r} =1$ and $\langle \sigma,\tau \rangle$ acts  irreducibly.
Then
\[ {\rm gcd}(n^d-1,p^\ell-1, n^{s}-p^{r}) = \sum_\delta \delta\sum_{(\sigma,\tau)\in I_{\delta}}  L_{\delta}^{\sigma,\tau}(n,p) \]
\end{proposition}
\begin{proof}
The proof is essentially identical to Proposition \ref{prop:orbitcount1}.
\end{proof}

\begin{remark}
The principal of inclusion/exclusion along with the above proposition allows one to find a formula for $\delta L_{\delta}^{\sigma,\tau}(n,p)$ whenever $\langle \sigma,\tau \rangle$ acts irreducibly.

Recall the sets $I_\delta$ essentially classify quotients of the group 
\[ \langle \sigma,\tau\;| \sigma^d = \tau^\ell = \sigma^{s}\tau^{-r} = 1 \rangle \] where the permutation action for the quotient is natural action of multiplication on itself.
As noted previously, the most natural cases to consider are those where $\langle \sigma,\tau \rangle = \langle \sigma \rangle$ so that $\tau = \sigma^s$ and $\ell = d/{\rm gcd}(d,s)$.
In this case one only needs to consider the inclusion/exclusion over $d' | d$ as all the quotients take the form
\[ \langle \sigma,\tau\;| \sigma^{d'} = \tau^\ell = \sigma^{s}\tau^{-r} = 1 \rangle\]
for some $d'|d$.
Finally, we point out that if 
\[ \langle \sigma,\tau\;| \sigma^{d'} = \tau^\ell = \sigma^{s}\tau^{-r} = 1 \rangle\]
is isomorphic to 
\[ \langle \sigma,\tau\;| \sigma^{d'} = \tau^{\ell'} = \sigma^{s}\tau^{-r} = 1 \rangle\]
under the map taking $\sigma$ to $\sigma$ and $\tau$ to $\tau$ then
\[ {\rm gcd}(n^{d'}-1,p^\ell-1, n^{s}-p^{r})= {\rm gcd}(n^{d'}-1,p^{\ell'}-1, n^{s}-p^{r}).\]
and similarly if rather than changing $\ell$ we replace $(s,r)$ by an equivalent $(s',r')$ which happen to give an isomorphic group.
\end{remark}

The following Theorem, combined with the previous propositions allows one to find formulas for any of
\[ L_d(n),\quad L_d^\pm(n), \quad L_d^{\sigma}(n) \]
for arbitrary $d$. 
\begin{theorem}\label{thm:monierfactors}
We may find the following bounds and formulas:
\begin{enumerate}
    \item 
   Suppose $\sigma$ is a $d$-cycle then
    \[ L_d^{\sigma,+}(n,p) + L_d^{\sigma,-}(n,p)\leq \frac{1}{\delta}{\rm gcd}(n^d-1,p^d-1). \]
    \item 
    Suppose $\sigma$ is a $d$-cycle and $\tau=\sigma^s$ has order $k|d$ then
      \[ L_d^{\sigma,\tau}(n,p) = \frac{1}{d}\sum_{o|d} \mu(o){\rm gcd}(n^{d/o}-1,p^k-1,n^s-p) \]
      where here $\mu$ is the Mobius function.
      

        \item Suppose $\sigma,\tau$ act transitively on $\{1,\ldots, \delta\}$ with $d = {\rm ord}(\sigma)$, $\ell = {\rm ord}(\tau)$ and minimal relation $\sigma^{d/t} = \tau^{a\ell/t}$ then we have $L_{\delta}^{\sigma,\tau}(n,p)$ equals
    \[   \frac{1}{\delta}\left(
    \sum_{\substack{o_1| d,\\ o_2|\ell,\\o_3|{\rm gcd}(d,\ell)/t}} \mu(o_1)\mu(o_2)\mu(o_3) {\rm gcd}(n^{d/o_1}-1,p^{\ell/o_2}-1, n^{d/o_3t} - p^{a\ell/o_3t})\right) \]
    \end{enumerate}
\end{theorem}

\begin{proof}
Part (1) is an immediate consequence of Proposition \ref{prop:orbitcount1} and the definitions of $ L_d^{\sigma,+}(n,p)$ and $ L_d^{\sigma,-}(n,p)$.

Part (2) we have from Proposition \ref{prop:orbitcount2}, and the classification of quotients of a cyclic group, that
\begin{align*}
 &\frac{1}{d}\sum_{o|d} \mu(o){\rm gcd}(n^{d/o}-1,p^k-1,n^s-p)\\
 &\quad= \frac{1}{d}\sum_{o|d} \mu(o) \sum_{m|(d/o)}\frac{d}{mo} \sum_{(\sigma',\tau')\in I_{d/mo}}  L_{d/mo}^{\sigma',\tau'}(n,p)\\
 &\quad= \frac{1}{d}\sum_{o|d} \mu(o) \sum_{m|(d/o)}\frac{d}{mo} L_{d/mo}^{\langle \sigma',\tau'| {\sigma}'^{d/mo}={\tau'}^k={\sigma'}^s{\tau'}^{-1} = 1\rangle}(n,p)\\
 &\quad= \sum_{t|d}\frac{\left(\sum_{o|t} \mu(o) \right)}{t}  L_{d/t}^{\langle \sigma',\tau'| {\sigma}'^{d/t}={\tau'}^k={\sigma'}^s{\tau'}^{-1} = 1\rangle}(n,p)\\
 &\quad = L_{d}^{\sigma,\tau}(n,p) 
\end{align*}
where in the forth step we took $t=mo$ and in the final step used that unless $t=1$ we have $\displaystyle\sum_{o|t} \mu(o) = 0$. Note that a subtle key to the proof is that the second line the $I_\delta$ which are associated to the different gcd's depend only on $\delta$, as $k$ and $s$ remain fixed, so that there is no abuse of notation arising in the fact that $I_\delta$ as introduced in Proposition \ref{prop:orbitcount2} depends on $d$, $\ell$, $r$, and $s$. In the third step it is then key that each of these sets are singletons where again $k$ and $s$ can be kept fixed by the classification of quotients.

The proof of part (3) is similar, but longer to write out in detail. The most relevant step is that
    \[   \frac{1}{\delta}\left(
    \sum_{\substack{o_1| d,\\ o_2|\ell,\\o_3|{\rm gcd}(d,\ell)/t}} \mu(o_1)\mu(o_2)\mu(o_3) {\rm gcd}(n^{d/o_1}-1,p^{\ell/o_2}-1, n^{d/o_3t} - p^{a\ell/o_3t})\right) \]
    becomes
\[ \frac{1}{\delta} \sum_{\delta'|\delta} \sum_{(\sigma',\tau')\in I_{\delta'}} \sum_{ o_1 | \frac{{\rm ord}(\sigma)}{{\rm ord}(\sigma')}}  \sum_{ 
o_2 | \frac{{\rm ord}(\tau)}{{\rm ord}(\tau')}} \sum_{  
o_3 | \frac{{\rm ord}(\rho)}{{\rm ord}(\rho')} 
} \mu(o_1)\mu(o_2)\mu(o_3) \delta'L_{\delta'}^{\sigma',\tau'}(n,p)
\]
where $\rho = \sigma^{d/{\rm gcd}(d,\ell)}\tau^{-a\ell/{\rm gcd}(d,\ell)}$ and $\rho' = (\sigma')^{d/{\rm gcd}(d,\ell)}(\tau')^{-a\ell/{\rm gcd}(d,\ell)}$.
This in turn simplifies to
\[ L_\delta^{\sigma,\tau}(n,p) \]
as in the simpler case.
\end{proof}

We give a concrete illustration of part (2) in the above.
\begin{example}
    Suppose $\sigma$ is a $30$-cycle and $\tau=\sigma^3$ so that the order of $\tau$ is $10$. Here we want to find $L_{30}^{\sigma,\tau}(n,p)$. 

    We are thus building polynomials whose roots $\alpha_i$ have ${\rm ord}(\alpha_i) \mid n^{30}-1$ and ${\rm ord}(\alpha_i) \mid p^{10}-1$. On the other way, we know $\tau=\sigma^3$, this means for all $\alpha_i$ $\alpha_i^p=\tau (\alpha_i)=\sigma^3(\alpha_i)= \alpha_i^{n^3}$. Thus ${\rm ord}(\alpha_i) \mid n^{3}-p$. Therefore, we have ${\rm gcd}(n^{30}-1, p^{10}-1, n^3-p)$ elements. Since $\alpha_i$'s are conjugates then choosing one of them is enough, so we divides it by $\frac{1}{30}$.
    
    On the other hand, we know ${\rm ord}(\alpha_i) \nmid n^{j}-1$ for any proper divisor $j$ of $30$, so $j=1,2,3,5,6,10,15$. When we subtract the terms ${\rm gcd}(n^{15}-1, p^{10}-1, n^3-p)$, ${\rm gcd}(n^{10}-1, p^{10}-1, n^3-p)$, and ${\rm gcd}(n^{6}-1, p^{10}-1, n^3-p)$; we actually subtracts those $\alpha_i$'s that have order $2$, $3$ and $5$ twice. We need to thus compensate by re-adding corresponding terms for $n^2-1$, $n^3-1$ and $n^5-1$. This in turn causes an over count for $n-1$. Hence we have 
    \begin{align*}
     L_{30}^{\sigma,\tau}(n,p)&= L_{30}^{\sigma,\sigma^3}(n,p)\\
     &=\frac{1}{30}\bigg({\rm gcd}(n^{30}-1, p^{10}-1, n^3-p)-{\rm gcd}(n^{15}-1, p^{10}-1, n^3-p)\\
       &\quad\qquad-{\rm gcd}(n^{10}-1, p^{10}-1, n^3-p)-{\rm gcd}(n^{6}-1, p^{10}-1, n^3-p)\\
       &\quad\qquad+ {\rm gcd}(n^{5}-1, p^{10}-1, n^3-p)+{\rm gcd}(n^{3}-1, p^{10}-1, n^3-p)\\
       &\quad\qquad+{\rm gcd}(n^{2}-1, p^{10}-1, n^3-p)-{\rm gcd}(n-1, p^{10}-1, n^3-p)\bigg)\\
       &= \frac{1}{30}\sum_{o|30} \mu(o){\rm gcd}(n^{30/o}-1,p^{10}-1,n^3-p).
    \end{align*}
    
\end{example}

Theorem \ref{thm:monierfactors} above massively generalize the special cases which have previously been studied in
\cite{Monier}, \cite{FioriShallue}, and \cite{Hiva} for degrees $1$, $2$, and $3$ respectively.
\begin{theorem}
 Suppose $p^r|\!|n$.
 In the cases of $d$ being respectively $1$, $2$, and $3$ we have
 \begin{itemize}
 \item For degree $1$ we have
\begin{align*} 
 L_1(n,p) &= {\rm gcd}(n-1,p-1)
 \end{align*}
\item For degree $2$ we have
\begin{align*} 
    L_2^{(1,2),(1,2)}(n,p) = L_2^{(1,2),-}(n,p) &= \frac{1}{2}({\rm gcd}(n^2-1,p^2-1,n-p) - {\rm gcd}(n-1,p-1))\\
    L_2^{(1,2),(1)(2)}(n,p)= L_2^{(1,2),+}(n,p) &= \frac{1}{2}({\rm gcd}(n^2-1,p-1) - {\rm gcd}(n-1,p-1))
\end{align*}
\item For degree $3$ we have
\begin{align*} 
    L_3^{(1,2,3),(1,2,3)}(n,p) &= \frac{1}{3}({\rm gcd}(n^3-1,p^3-1,n-p) - {\rm gcd}(n-1,p-1))\\
    L_3^{(1,2,3),(1,3,2)}(n,p) &= \frac{1}{3}({\rm gcd}(n^3-1,p^3-1,n^2-p) - {\rm gcd}(n-1,p-1))\\
    L_3^{(1,2,3),(1)(2)(3)}(n,p) &= \frac{1}{3}({\rm gcd}(n^3-1,p-1) - {\rm gcd}(n-1,p-1))
\end{align*}
and consequently we see:
\[ L_3^{(1,2,3),-}(n,p) = 0\]
and
\[ L_3^{(1,2,3),+}(n,p) = \frac{1}{3}\left(\sum_{j=0}^2{\rm gcd}(n^3-1,p^3-1,n^j-p) \right) - {\rm gcd}(n-1,p-1) \]
 \end{itemize}
\end{theorem}

\section{Lower Bounds}\label{sec:lower}

In this section we will obtain conditional lower bounds for
\[ \sum_{n<x} L_d^{\sigma}(n) \]
In the case where $\sigma$ is a $d$-cycle. We first recall some key number theoretic inputs we shall need.

\subsection{Number Theory Background}
\begin{definition}
For each value $x\geq 1$ denote by $M(x)$ the least common multiple of all positive integers up to $\frac{\log x}{\log \log x}$.
\end{definition}

The following is well known.
\begin{proposition}
    We have $M(x) = x^{o(1)}$ for large $x$. 
\end{proposition}

\begin{definition}
In our construction, we shall need to use prime numbers where certain polynomial values of them are smooth. The existence of such primes has been well studied, see for example \cite{DMT}. Now we introduce the following set
\begin{equation}
    \Psi_F^{*}(x, y)=\left|\left\{p \leqslant x: P^{+}(F(p)) \leqslant y\right\}\right|,
\end{equation}
where $P^+(m)$ denotes the largest prime factor of $m$,
with the convention that $P^+(1) = 1$, and $F(x)$ is a polynomial with integer coefficients. Let $g$ be the largest degrees of the irreducible factors of $F$ and let $k$ be the number of distinct irreducible factors of $F(X)$ of degree $g$. 
\end{definition}

\begin{definition}
Let $F(x)$ be a polynomial with integer coefficients. The following set is denoted by $P_{\alpha,F}(x)$,
\[ \left\{p< (\log x)^{\alpha} \text{ such that } F(p) \mid M(x)\right\}.\]
\end{definition}

The following Lemma has been used without proof, or being explicitly stated, in several results. We fill in the proof here so as to complete the literature.
\begin{lemma}\label{A}
Let $\alpha>1/2$ such that $\Psi_F^\ast(\log(x)^\alpha, \log(x)/\log\log(x)) \sim \log(x)^{\alpha-o(1)}$ then $P_{\alpha,F}(x) \sim \log(x)^{\alpha-o(1)}$.
\end{lemma}
\begin{proof}
We will show $\Psi_F^\ast(\log(x)^\alpha, \log(x)/\log\log(x))$ and $ \mid P_{\alpha,F}(x) \mid $ are of the same order by considering their difference.
The difference between $P_{\alpha,F}(x)$ and the set $\{p<(\log x)^{\alpha} \mid \forall q|F(p), q \leq \frac{\log x}{\log\log x} \}$ is the set 
\[ A = \left\{ p < (\log x)^{\alpha} \mid  \forall q \mid F(p), q < \frac{\log x}{\log\log x} \text{ and } \exists q^{\beta_i} |\!| F(p),\,  q^{\beta_i} \nmid M \right\}, \]
where $q^{\beta_i}$ is the largest prime power divisor of $F(p)$.
Now we want to bound the size of the set $A$. We introduce the larger set.
\[ B=\left\{ n < \log(x)^\alpha \bigg| \exists q^r, q^r | F(n), q< \frac{\log(x)}{\log\log(x)},\, q^r >  \frac{\log(x)}{\log\log(x)}\right\}.\]
For convenience, we can bound the above set by counting this separately for each fixed $q$. Thus we define the set 
 \[ B_q = \left\{ n < \log(x)^\alpha \bigg| \exists r,\,  q^r | F(n),\, q< \frac{\log(x)}{\log\log(x)},\, q^r >  \frac{\log(x)}{\log\log(x)}\right\}. \]
For $q < \sqrt{\log(x)/\log\log(x)}$, we have 
\[ \mid B_q \mid \leq {\rm Deg}(F) \frac{\log(x)^\alpha}{q^r} <{\rm Deg}(F) (\log x)^{\alpha-1}\log\log x.\]
For $q\ge\sqrt{\log(x)/\log\log(x)}$, then we have
\[ \mid B_q \mid \leq {\rm Deg}(F)\frac{(\log x)^{\alpha}}{q^2}.\]
Now, we take a summation over $1<q<\log(x)/\log\log(x)$
\begin{align*}
\mid A \mid \leq \mid B \mid &\leq \sum_q |B_q|\\
&\ll (\log x)^{\alpha-\frac{1}{2}}(\log\log x+\sqrt{\log\log x})\\
&\ll (\log x)^{\alpha-\frac{1}{2}-o(1)}.
\end{align*}
Thus provided $\alpha>1/2$ we have
\[ P_{\alpha,F}(x)\sim \Psi_F^\ast(\log(x)^\alpha, \log(x)/\log\log(x)) \]
from which one obtains the result.
\end{proof}

\begin{theorem}\cite[Theorem 1.2]{DMT}\label{si}
Let $F(x)\in \mathbb{Z}[x]$. Let $g$ be the largest degree of an irreducible factor and $k$ the number of factors of degree $g$.
If $g = k = 1$ then assume $F(0)\neq 0 $. Finally, let $\varepsilon$ be a positive real number. Then the estimate
\begin{equation}\label{s}
    \Psi_F^*(x, y) \asymp \frac{x}{\log x}
\end{equation}
holds for all large $x$ provided $y \geqslant x^{g+\varepsilon-1 / 2 k}$.
\end{theorem}
Improvements to the above in special cases are known, see for example \cite{BALOG1983-1984}, \cite{bakerharman}, \cite{Banks_Shparlinski_2007} and \cite{lichtman2022primesarithmeticprogressionslarge} which have progressively provided information in the case of degree $1$ and \cite{dartyge1996entiers} which provides information about the specific polynomial $F(n)=n^2+1$.

\begin{proposition}
For all $\alpha \leq \frac{2}{2\phi(d)-1}$, where $F(x)=x^d-1$, we have  
$$\left|P_{\alpha,F}(x)\right| \geq(\log x)^{\alpha-o(1)}$$
as $x \rightarrow$ $\infty$. Here $\phi(d)$ is the Euler function.
\end{proposition}
\begin{proof}
By Lemma \ref{A} and Theorem \ref{si}, we have $\alpha^{-1} \geq \phi(d)-1/2$.
\end{proof}
\begin{conjecture}\cite[p. 3]{DMT}\label{l:dm}
The asymptotic \eqref{Psi} is satisfied for any positive $\alpha$.
\begin{equation}\label{Psi}
\Psi_F^*(x, x^{\alpha}) \asymp_{F, \alpha} \frac{x}{\log x}
\end{equation}
\end{conjecture}
We recall the well known result of Linnik's theorem originally from \cite{Linnik}. The most recent explicit version of this is \cite{Xylouris}.
\begin{theorem}[Linnik’s theorem]
    Let $P(a; q)$ be the least prime in an arithmetic progression $a \bmod q$ where a and q are co prime positive integers. Then there exists an effectively computable constant $L>0$ such that
    $$P(a; q) < q^L.$$   
\end{theorem}

\subsection{Construction}

We are now in a position to give the key constructions.

\begin{definition}\label{pl}
For each value $x$ and for each $\alpha >0$ we define the set $P_{\alpha,F}(x, a)$ by,
\[P_{\alpha,F}(x, a)=\left\{ p < (\log x)^{\alpha} \mid p=a  \bmod{F(p)}, \text{ and }  F(p)\mid M(x) \right\},\]
where $F(a)=0 \pmod{M(x)}$ for $a \in \mathbb{N}$. 
\end{definition}

\begin{proposition}\label{1}
 Given $\alpha >0$ such that $|P_{\alpha,F}(x)| \geq(\log x)^{\alpha-o(1)}$ as $x \rightarrow \infty$, then there exists integer $a \neq 1$ such that as $x \rightarrow \infty$ we have
$$
\left|P_{\alpha,F}\left(x,a\right)\right| \geq(\log x)^{\alpha-o(1)} .
$$
\end{proposition}
\begin{proof}
We know that each prime $p \in P_{\alpha,F}\left(x,a\right)$ is also in $P_{\alpha,F}(x)$. Conversely, each prime $p \in P_\alpha(x)$ is also in $P_\alpha\left(x,a\right)$ for all $a$ satisfying  $p=a \bmod{F(p)}$. By Definition \ref{pl}, we know
$$F(a)=0 \bmod{F(p)} \text{ and } F(a)=0 \bmod{M(x)},
$$
so by the Chinese Reminder Theorem we have 
$$F(a)=0 \bmod{M^{'}},$$
where $M^{'}$ is the largest divisor of $M$ with ${\rm gcd}(M{'},F(p))=1$. 
$$A=\left\{a \in \mathbb{Z}/M\mathbb{Z} \mid F(a)=0\pmod{M}\right\},$$
$$A(p) = \left\{ a \in \mathbb{Z}/F(p)\mathbb{Z} \mid F(a)=0 \pmod{F(p)}\right\}.$$
Since there exists at least one value that is as large as the average, we can say that there is a value $a^{\prime}$ such that
$P_{\alpha,F}\left(x,a'\right) \ge \frac{1}{\mid A\mid} \sum_a \mid P_{\alpha,F}\left(x,a\right)\mid$. So we have 
$$
\begin{aligned}
P_{\alpha,F}(x,a') &\ge \frac{1}{\mid A\mid} \sum_a \mid P_{\alpha,F}\left(x,a\right)\mid\\
&=\sum_{p \in P_\alpha(x)} \frac{\mid\{ \text{ solutions to } F(a) = 0 (\bmod{M^{'}})\}\mid}{{\mid A\mid}}\\
& \geq {\rm deg}(F)^{-\omega_{\max}(F(p))}(\log x) ^{\alpha-o(1)}\\
& \geq (\log x) ^{\alpha-o(1)}.
\end{aligned}
$$
Notice that $|A(p)| \leq ({\rm deg}(F))^{\omega(F(p))}$.  
Since the number of prime divisors of $F(p)$ is $\omega(F(p))$. Then the number of possible solutions for equation $F(a)=0\pmod{F(p)}$ are at most $({\rm deg}(F))^{\omega(F(p))}$. Now by \cite[Section 22.10]{Hardy} we know that when $p < (\log x)^{\alpha}$ we have
$${\omega_{\max}(F(p))}\leq(1+o(1)) \frac{\log(\log x^{\alpha})}{\log (\log(\log x^{\alpha}))} \leq (\log x)^{o(1)},$$ where $\omega_{\max}(F(p))$ is the maximum number of primes that divide a number of size $F(p)$. 
\end{proof}

\begin{definition}\label{k}
Fix $d>0$ and $\sigma$ a $d$-cycle.
Let $0<\epsilon<\alpha-1$ and for all $x>0$ let
\begin{itemize}
\item $F(n) = n^d-1$.
\item $P_\alpha(x,a) = P_{\alpha,F}(x,a)$ 
\item Let $\ell_1$ and $\ell_2$ be the two smallest primes larger than $\log(x)/\log\log(x)$ which are congruent to $1\pmod{d}$.
\item Write $M=\ell_1\ell_2 M(x)$.
\item $k_\alpha(x)=\left\lfloor\frac{\log x-2 L \log M}{\alpha \log \log x}\right\rfloor$,

\item 
$S_{\alpha, \epsilon}(x,a)$ be the set of integers $s$ which are the product of $k_\alpha(x)$ distinct elements from
$$
P_\alpha(x,a) \backslash P_{\alpha-\epsilon}(x,a).
$$
That is, 
$$S_{\alpha, \epsilon}(x,a)=\{s | s=\prod_{i=1}^{k_\alpha(x)} p_i , p_i \in  P_\alpha(x,a) \backslash P_{\alpha-\epsilon}(x,a) \hspace{1mm}\text{and}\hspace{1mm}, p\text{'s are distinct}  \}.$$
\end{itemize}

\end{definition}

The following are easy consequences of the construction, the proofs being completely analogous to \cite{ErdosPomerance}, \cite{FioriShallue}, and \cite[Lemma 4.30 and Proposition 4.33]{Hiva} hence we omit the proofs.
\begin{lemma}
Given $\alpha>1$, the elements $s$ of $S_{\alpha, \epsilon}(x,a)$ all satisfy
$$
x^{-\epsilon \alpha^{-1}(1+o(1))}\frac{x}{M^{2L}} \leq s<\frac{x}{M^{2 L}}
$$
as $x \rightarrow \infty$.
Here the constants in the $o(1)$ are uniform in $\epsilon$.
\end{lemma}

\begin{lemma}\label{lem:countS}
Given $\alpha >1$ for which $P_{\alpha}(x,a)> (\log x)^{\alpha-o(1)}$, then
$$
\left|  S_{\alpha, \epsilon}(x,a) \right| \geq x^{1-\alpha^{-1}+o(1)}
$$
as $ x \rightarrow \infty$.

In the above the dependence of the $o(1)$ term on $\epsilon$ can be uniform in $\epsilon$ provided $\epsilon>\!\!> 1/\log\log(x)$.
\end{lemma}

The following is the key to the construction:
\begin{lemma}\label{strongconst}
Let $\sigma$ be a $d$-cycle. Let $L$ be an upper bound for Linnik's constant. Given any element $s$ of $S_{\alpha, \epsilon}(x,a)$ where  $\alpha>1$ then there exists a number $M<q<M^{2 L}$ such that $n = sq$ and we have 

\begin{enumerate}
\item $n=a\pmod{M(x)\ell_1\ell_2}$, 
\item $\operatorname{gcd}(q, s)=1$,
\item $\prod_{p \mid q} L_{d}^{\sigma,\sigma}(n,p)>0$, and
\item $L_d^{\sigma}(n) = x^{d-d\frac{\epsilon}{\alpha}-\frac{\epsilon}{\alpha}\operatorname{o}(1)}$.\\ Here the constants in the $o(1)$ term are uniform in $\epsilon$.
\end{enumerate}
\end{lemma}
\begin{proof}
Let $\Phi_d(x)$ denote the $d$th cyclotomic polynomial.
Then fix $a\pmod {\ell_1}$ and $a\pmod{\ell_2}$ solutions to $\Phi_d(x)=0$. 

Recall we write $M=\ell_1\ell_2M(x)$.
Now we construct $q$ as the product of two primes $q_1$ and $q_2$, which are chosen as the smallest primes  satisfying certain conditions modulo $M$, as given below. 
\begin{equation}\label{l1}
q_1 \equiv a \hspace{1mm} (\operatorname{mod} \ell_1 )\quad \quad q_2 \equiv s^{-1} \hspace{1mm} (\operatorname{mod} \ell_1) \quad \quad q_1 \equiv s^{-1} \hspace{1mm} (\operatorname{mod} M(x)) 
\end{equation}
\begin{equation}\label{l2}
    q_2 \equiv a \hspace{1mm}(\operatorname{mod} \ell_2) \quad \quad q_1 \equiv s^{-1} \hspace{1mm}(\operatorname{mod} \ell_2) \quad \quad q_2 \equiv a \hspace{1mm}(\operatorname{mod} M(x)).
\end{equation}
This system has a solution by the Chinese Reminder Theorem. Additionally, by Linnik's theorem, we can bound $q=q_1q_2 < M^{2L}$. Setting $n=sq$ these conditions give us the first condition. The second condition is a consequence of the sizes of $q_1$ and $q_2$.

For the third condition the congruence conditions above give us
\[ \ell_1 \big| {\rm gcd}(\Phi_d(n),\Phi_d(q_1),n-q_1) \big|  {\rm gcd}(F(n),F(q_1),n-q_1) 
\]
and
\[ \ell_2 \big| {\rm gcd}(\Phi_d(n),\Phi_d(q_2),n-q_2) \big|  {\rm gcd}(F(n),F(q_2),n-q_2) 
\]
However, because $a$ is a root of $\Phi_d(x)$ it is not a root of $\Phi_{d'}(x)$ for any $d'|d$.
It follows that $ \Phi_{d'}(n) \neq 0 \pmod{\ell_1}$ and  $\Phi_{d'}(n) \neq 0 \pmod{\ell_2}$ and hence that 
$\ell_1 \nmid {\rm gcd}(n^{d'}-1,F(q_1),n-q_1) $ and $\ell_2 \nmid {\rm gcd}(n^{d'}-1,F(q_2),n-q_2)$ from which it follows by
Theorem \ref{thm:monierfactors} that
\[  L_{d}^{\sigma,\sigma}(n,q_1)L_{d}^{\sigma,\sigma}(n,q_2) > 0. \]

For the forth condition of the theorem we have for all $p|s$ that $n=a\pmod{M}$ hence
$n=a\pmod{F(p)}$ which immediately implies $ F(p) \big| {\rm gcd}(F(n),F(p),n-p)$.
\[ F(p) = {\rm gcd}(F(n),F(p),n-p). \]
Since every other term in the formula from Theorem \ref{thm:monierfactors} is bounded by $p^{d'}-1$ for some $d'|d$, for $x$ sufficiently large and hence $p$ sufficiently large we have $L_{d}^{\sigma,\sigma}(n,p) = p^{d+o(1)}$.
We thus have, using also (3), that
\[ L_{d}^{\sigma}(n) \ge \prod_{p|s} L_{d}^{\sigma,\sigma}(n,p) = \prod_{p|s} p^{d-o(1)} = s^{d-o(1)} = x^{d-d\frac{\epsilon}{\alpha} - o(1)}.
 \]
 For a more detailed verification of the final calculations for (4) in the above see \cite[Lemma 4.35]{Hiva}.
\end{proof}

\begin{remark}
The above Lemma constructs many examples where for all $p|n$ we have $\sigma=\tau$. Once can construct equal numbers for which $\sigma^t=\tau$. This would only effect the counts by a factor of $x^{o(1)}$ so would not improve our final lower bound.
\end{remark}

The following is an immediate consequence, the proof being identical to that of \cite{ErdosPomerance}, \cite{FioriShallue}, and \cite{Hiva}.
\begin{theorem}\label{thm:lbconjstrong}
Let $\sigma$ be a $d$-cycle.
For any value of $\alpha > 1 $ satisfying Proposition \ref{1} we have the asymptotic inequality
    \[ \sum_{n<x}L_d^{\sigma}(n) \geq x^{d+1-\alpha^{-1} -o(1)}\]
    as $x \rightarrow \infty$.
\end{theorem}
\begin{proof}
This is essentially immediate from Lemmas \ref{lem:countS} and \ref{strongconst} with the only technical detail being that to absorb the $\epsilon$ into the $o(1)$ one takes $\epsilon = 1/\log\log(x)$ and uses that the existing $o(1)$ terms were uniform for $\epsilon$ of this size.
\end{proof}

Noting that the above theorem depends on a conjecture about which we have no clear approach for $d\ge 2$. We thus present a much weaker conjecture which allows us to get a weaker result.

\begin{definition}\label{p_2}
Fix $\alpha > 0$ and a prime number $\ell$ define
\[ P^{\prime}_{\alpha,\ell}(x)=\{ p < (\log x)^{\alpha} \mid p=1 \mod{\ell}\text{ and } p-1 \mid M(x) \} \]
\end{definition}

In contrast to the conjecture about $P_{\alpha,F}(x)$ it is expected the following could be established with existing techniques, however, doing so would be well outside the scope of this paper.
\begin{conjecture}\label{co}
For all primes $\ell$ and $1 < \alpha \leq 2$ we have that 
\[ \left|P^{\prime}_{\alpha,\ell}(x)\right| \geq(\log x)^{\alpha-o(1)} \] as $x \rightarrow$ $\infty$.
\end{conjecture}
Though we are unaware of results establishing the above we note that the analogous problem for smooth shifted integers, rather than primes, has at least been studied, see \cite{fouvry1991entiers}.

What we actually need is the following slightly weaker conjecture.
\begin{conjecture}\label{co2}
Fix $d$.
There exists $\ell = 1 \pmod{d} $ and $\alpha > 1$ such that 
\[ \left|P^{\prime}_{\alpha,\ell}(x)\right| \geq(\log x)^{\alpha-o(1)} \] as $x \rightarrow$ $\infty$.
\end{conjecture}

We now present a construction analogous to the previous one.
\begin{definition}
Let $S_{\alpha, \epsilon, \ell}(x)$ be the set of integers $s$ which are the product of $k_\alpha(x)$ distinct elements from
\[
P^{\prime}_{\alpha,\ell}(x) \backslash P^{\prime}_{\alpha-\epsilon,\ell}(x). 
\]
\end{definition}

Finally, we have the following weakened version of Lemma \ref{strongconst}
\begin{lemma}\label{qss}
Fix $d$ and $F(n)=n^d-1$ as before.
Let $\ell = 1\pmod{d}$ and $\alpha >0$ be such that $\left|P^{\prime}_{\alpha,\ell}(x)\right| \geq(\log x)^{\alpha-o(1)}$.
Let $L$ be an upper bound for Linnik's constant. 
Given any element $s$ of $S_{\alpha, \epsilon,  \ell}(x)$ there exists a number $q<M^{2 L}$ such that with $n=sq$
\begin{enumerate}
\item $F(n) = 0 \pmod{M}$,
\item $\ell \nmid n^{d'}-1$ for any $d'|d$ with $d'\neq d$.
\item $\operatorname{gcd}(q, s)=1$,
\item $\prod_{p \mid q} L_{d}^{\sigma,\sigma}(n,p)>0$, and
\item $L_d^{\sigma}(n) \geq x^{1-\frac{\epsilon}{\alpha}-\operatorname{o}(1)}
$
as $x \rightarrow \infty$.

\end{enumerate}
\end{lemma}
\begin{proof}
The proof is very similar to the previous.
Let $a$ be any root of $F(a)=0 \pmod{\ell_1\ell_2M(x)}$ such that
\[ a \neq 1\pmod {\ell_1} \quad a\neq 1 \pmod{\ell_2} \quad \Phi_d(a) = 0 \pmod{\ell} . \] 

Now we construct $q$ as the product of two primes $q_1$ and $q_2$, which are chosen as the smallest primes  satisfying certain conditions modulo $M=\ell_1\ell_2M(x)$, as given below. 
\begin{equation}\label{l1}
q_1 \equiv a \hspace{1mm} (\operatorname{mod} \ell_1 )\quad \quad q_2 \equiv s^{-1} \hspace{1mm} (\operatorname{mod} \ell_1) \quad \quad q_1 \equiv s^{-1} \hspace{1mm} (\operatorname{mod} M(x)) 
\end{equation}
\begin{equation}\label{l2}
    q_2 \equiv a \hspace{1mm}(\operatorname{mod} \ell_2) \quad \quad q_1 \equiv s^{-1} \hspace{1mm}(\operatorname{mod} \ell_2) \quad \quad q_2 \equiv a \hspace{1mm}(\operatorname{mod} M(x)).
\end{equation}
Points (1), (2), (3), and (4) are as before.

For (5) the key is that we now have for each $p|s$
\[ \ell \big| {\rm gcd}(n^{d}-1,p-1) = p-1\]
but as $a$ is primitive $d$th root of unity modulo $\ell$, and $n=a\pmod{\ell}$, for any $d'|d$ with $d'\neq d$ we have
\[ \ell \nmid {\rm gcd}(n^{d'}-1,p-1). \]
Hence,
\[ {\rm gcd}(n^{d'}-1,p-1) < \frac{p-1}{\ell}. \]
The number of divisors of $d'$ is trivially bounded by $d\leq \ell-1$ so that from the formulas in Theorem \ref{thm:monierfactors} we obtain
\[ L_{d}^{\sigma,1}(n,p) \ge \frac{p-1}{\ell d} = p^{1-o(1)}. \]
The remainder of the proof is as before.
\end{proof}

The following result is then analogous, but weaker than the previous result.
\begin{theorem}\label{thm:lbconjweak}
For any values of $\ell$ and $\alpha>1$ satisfying Conjecture \ref{co2} we have the asymptotic inequality
 \[\sum_{n<x}L_d^{\sigma}(n) \geq x^{2-{\alpha}^{-1}-o(1)}\]
    as $x \rightarrow \infty$.
\end{theorem}

\begin{remark}
    If one establishes a hybrid of Conjectures \ref{l:dm} and \ref{co2} for the polynomials $\Phi_{d'}(x)$ with $\alpha>1$ then for any $d$ with $d'|d$ one can prove a lower bound of the form $x^{d'+1-{\alpha}^{-1}-o(1)}$.

    In the case of $d=d'$ one does not need the additional requirement from Conjecture \ref{co2} as in Theorem \ref{thm:lbconjstrong}.
\end{remark}

\section{Upper Bounds}\label{sec:upper}

In this section we obtain upper bounds on
\[ \sum_{n<x} L_d(n). \]
Our approach to upper bounds will differ from past approaches for Frobenius pseudoprimes from \cite{FioriShallue} and \cite{Hiva}. Rather, we prove a reduction to the case handled by \cite{ErdosPomerance}.

\begin{proposition}
If $(n,f)$ is a Frobenius pseudoprime then so is $(n,x-f(0))$.
\end{proposition}
\begin{proof}
Because $\phi_n$ permutes the $\alpha_i$ it is easy to see that it stabilizes the product $f(0)=\prod_{i=1}^d \alpha_i$. 
\end{proof}

\begin{remark}\label{rem:relative}
More generally, if $(n,f)$ is a Frobenius pseudoprime with $f=F_d$ of degree $d$ and $d' | d$ there is typically an associated Frobenius pseudoprime $(n,g)$ where $g$ has degree $d'$.
The roots of $g \pmod{p}$ are the  ``relative norms" of those of $f$.

More specifically, write $\displaystyle m=\sum_{i=0}^{k-1} n^{d'i}$ then if the roots of $f$ are
\[ \alpha,\alpha^n,\ldots,\alpha^{n^{d-1}} \]
then with $\beta=\alpha^m$ the roots of $g$ are
\[ \beta,\beta^n,\ldots,\beta^{n^{d'-1}}.\]

Note that
\[ \beta^{n^j} = (\alpha^{n^j})^m \]
and if $j_1=j_2\pmod{d'}$ then
\[ \beta^{n^{j_1}} = \beta^{n^{j_2}} \]
so that the roots of $g$ modulo $p$ are permuted by $\phi_n$.
Likewise, if $\Frob_p(\alpha) = \alpha^{n^j}$ then
\[ \Frob_p(\beta) = \beta^{n^j} \]
so that the roots of $g$ are also permuted by $\Frob_p$.

Finally, we construct $g \pmod{n}$ from $g\pmod{p}$ for each $p|n$ by way of the Chinese remainder theorem.

Note this construction  fails to produce a Frobenius pseudoprime if $g$ has repeat roots modulo $p$ for some $p|n$. This can occur if $\beta^{n^j} = \beta$ for some $j<d'-1$. Note also that attempting to loosen the condition around $f=F_d$ being degree $d$ adds complexity around the Jacobi' condition. As we shall note use this construction we leave verifying this condition in the presented case as an exercise.
\end{remark}

\begin{definition}
Define $\mathcal{F}_{d}(n)$ to be the set of $f\pmod n$ of degree $d$ where $(f,n) $ is a Frobenius pseudoprime.

Define $\mathcal{F}_{d,a}(n)$ to be the set of $f\pmod n$ of degree $d$ where $(f,n) $ is a Frobenius pseudoprime and $f(0)=a$.

Define $\mathcal{F}^{\sigma}_{d}(n)$ to be the set of $f\pmod n$ of degree $d$ where $(f,n) $ is a Frobenius pseudoprime, with cycle type $\sigma$.

Define $\mathcal{F}^{\sigma}_{d,a}(n)$ to be the set of $f\pmod n$ of degree $d$ where $(f,n) $ is a Frobenius pseudoprime and $f(0)=a$, with cycle type $\sigma$.
\end{definition}

\begin{theorem}\label{thm:upper}
We have
\[  \sum_{n<x} L_d(n) \leq x^{d-1} \sum_{n<x} L_1(n) \leq x^{d}\mathcal{L}(x)^{-1+o(1)}. \]
\end{theorem}
\begin{proof}
For all $d$, $a$, and $n$ we have the trivial bound
\[ \abs{ \mathcal{F}_{d,a}(n) } < n^{d-1}. \]
We thus have
\[  \sum_{n<x} L_d(n) = \sum_{n<x}\sum_{a\in \mathcal{F}_1(n)} \abs{ \mathcal{F}_{d,a}(n)} \leq \sum_{n<x} n^{d-1} L_1(n) \leq x^{d}\mathcal{L}(x)^{-1+o(1)} \]
which is the result.
\end{proof}

\begin{remark}
The above implies that the degree $d$ test is in some sense strictly better than the degree $1$ test. More generally Remark \ref{rem:relative} suggests that degree $d$ tests would be strictly better than the degree $d'$ test for any $d'|d$.
\end{remark}

The above motivates a study of $ \abs{ \mathcal{F}_{d,a}(n) }$ in the hopes of improving the result. Given $\sigma$ and $\tau$ with the usual parameters $\delta$, $d$, $\ell$, $r$ and $s$ for each $p|n$ define
\[ \mathcal{F}_{\delta,a}^{\sigma,\tau}(n,p) = \left\{ \alpha \in \overline{\mathbb{F}_p}^\times \;\bigg|\; \phi_n^{d}(\alpha)=\alpha\ ,{\rm Fr_p}^\ell(\alpha)=\alpha,\,\ \phi_n^{r}(\alpha)={\rm Fr}_p^s(\alpha),\, N_{\sigma,\tau}(\alpha)=a \right\} \]
where 
\[ N_{\sigma,\tau}(\alpha) = \prod_{\gamma\in\langle \sigma,\tau\rangle} \gamma(\alpha) \]
and for $\gamma = \sigma^x\tau^y$ we have 
\[ \gamma(\alpha) = \phi_n^x({\rm Fr}^y_p(\alpha)). \]
We note that $N_{\sigma,\tau}$ is well defined given the other conditions and is really just $\alpha^z$ for some $z$ depending on $\sigma$ and $\tau$ hence is a homomorphism.
\begin{proposition}
Suppose $d>1$ and $\sigma$ is a $d$-cycle.
For all $a$ with $\abs{\mathcal{F}^{\sigma,\tau}_{d,a}(n,p)} \neq 0$ we have $\abs{\mathcal{F}^{\sigma,\tau}_{d,a}(n,p)} = \abs{\mathcal{F}^{\sigma,\tau}_{d,1}(n,p)}$.
\end{proposition}
\begin{proof}
Fix any $\alpha \in \mathcal{F}_{d,a}^{\sigma,\tau}(n,p)$.
Now, define a map
\[ \mathcal{F}_{d,a}^{\sigma,\tau}(n,p) \rightarrow \mathcal{F}_{d,1}^{\sigma,\tau}(n,p) \]
by
\[ \beta \mapsto \beta/\alpha. \]
We easily see this has an inverse
\[ \beta \mapsto \beta\alpha. \qedhere\]
\end{proof}

\begin{remark}
It is important to note that counts for $\abs{\mathcal{F}^{\sigma,\tau}_{d,a}(n,p)} = \abs{\mathcal{F}^{\sigma,\tau}_{d,1}(n,p)}$ do not directly count Frobenius pseudoprimes. Aside from a need to divide by $\delta$ one would need an inclusion/exclusion formula and the behavior of these may differ for $a$ and $1$ because of the non-zero condition.

Despite this from the above it would be natural to try to bound the quantities
     \[ {\rm gcd}\left(\frac{n^d-1}{n-1},\frac{p^d-1}{p-1},n-p^s\right). \]
    which would be involved in formulas for $\abs{\mathcal{F}^\sigma_{d,1}(n)}$.

More specifically, it would be very natural to want to bound
\[ \sum_{n<x} \prod_{p|n} {\rm gcd}\left(\frac{n^d-1}{n-1},\frac{p^d-1}{p-1}\right) \]
or better yet
\[ \sum_{n<x} \prod_{p|n}  {\rm gcd}\left(\frac{n^d-1}{n-1},\frac{p^d-1}{p-1}\right)^{\!k} \]
so that one could use H\"older's inequality to improve Theorem \ref{thm:upper}.
\end{remark}

\subsection{Numbers with large prime factors}

We now show that numbers with large prime factors will tend not to have many Frobenius Liars.

\begin{proposition}
Suppose $\sigma$ is a $d$-cycle.
If $p|n$ and $p>n^{\frac{1}{d-1}}$ then 
\[ L^\sigma_d(n,p) < p^{d-2} \] hence 
\[ L^\sigma_d(n) < n^d/p^2.\]
If instead $p|n$ and $n^{\frac{1}{d-1}} >p>n^{\frac{1}{d+1}}$  then 
\[ L^\sigma_d(n,p) < p^{d} n/p^{d+1} < p^d\] and hence 
\[ L^\sigma_d(n) < n^{d+1}/p^{d+1} < n^d.\]
\end{proposition}
\begin{proof}
   Write $n=kp$ then notice that for each $s\in \{1\ldots,d-1\}$ we have
   \[ {\rm gcd}(n^d-1,p^d-1,n-p^s) = {\rm gcd}(n^d-1,p^d-1,k-p^{s-1}). \]
   Notice that $k-p^{s-1} < {\rm max}(k, p^{d-2}) = p^{d-2}$.
   If follows that $L^\sigma_d(n,p) < p^{d-2}$, hence $L^\sigma_d(n) < n^d/p^2$.
\end{proof}

\begin{remark}
One can strengthen the above if there are multiple large prime factors.
In particular, if the number of prime factors is small relative to the degree then the test would tend to be very effective.
\end{remark}

\subsection{Numbers with many prime factors}

\begin{proposition}
Suppose $\sigma$ is a $d$-cycle with $d$ prime.
Fix  $s_1\neq s_2 \in \{1\ldots,d-1\}$
Then 
\[ \frac{{\rm gcd}(n^d-1,p^d-1,n-p^{s_1}){\rm gcd}(n^d-1,p^d-1,n-p^{s_2})}{{\rm gcd}(n-1,p-1)} < {\rm gcd}(n^d-1,p^d-1). \]
It follows that $L^\sigma_d(n) < \frac{n^d}{d^{w(n)}}$.
\end{proposition}
\begin{proof}
    Both ${\rm gcd}(n^d-1,p^d-1,n-p^{s_1})$ and ${\rm gcd}(n^d-1,p^d-1,n-p^{s_1})$ describe subgroups of $\overline{\mathbb{F}_p}^\times$ contained in ${\rm gcd}(n^d-1,p^d-1)$. Their intersection is ${\rm gcd}(n-1,p-1)$. This gives the first claim.   
\end{proof}

\begin{remark}
The above result can be extended to include products with more $s_i$, the situation where $d$ is composite could be treated as well at the expensive of describing the intersections.
\end{remark}

\bibliographystyle{alpha}
\bibliography{refs.bib}

\begin{thebibliography}{BFWJ21}

\bibitem[Bal84]{BALOG1983-1984}
Antal Balog.
\newblock p+a without large prime factors.
\newblock {\em Seminaire de Théorie des Nombres de Bordeaux}, 13:1--6, 1983-1984.

\bibitem[BFWJ21]{BFW}
Robert Baillie, Andrew Fiori, and Samuel Wagstaff~Jr.
\newblock Strengthening the {B}aillie-{PSW} primality test.
\newblock {\em Mathematics of Computation}, 90(330):1931--1955, 2021.

\bibitem[BH98]{bakerharman}
R~Baker and Glyn Harman.
\newblock Shifted primes without large prime factors.
\newblock {\em Acta Arithmetica}, 83(4):331--361, 1998.

\bibitem[BS07]{Banks_Shparlinski_2007}
William~D. Banks and Igor~E. Shparlinski.
\newblock On values taken by the largest prime factor of shifted primes.
\newblock {\em Journal of the Australian Mathematical Society}, 82(1):133–147, 2007.

\bibitem[Dar96]{dartyge1996entiers}
C{\'e}cile Dartyge.
\newblock Entiers de la forme n 2+ 1 sans grand facteur premier.
\newblock {\em Acta Mathematica Hungarica}, 72:1--34, 1996.

\bibitem[DTM02]{DMT}
C{\'e}cile Dartyge, G{\'e}rald Tenenbaum, and Greg Martin.
\newblock Polynomial values free of large prime factors.
\newblock {\em Periodica Mathematica Hungarica}, 43:111--119, 2002.

\bibitem[EP86]{ErdosPomerance}
Paul Erd{\H{o}}s and Carl Pomerance.
\newblock On the number of false witnesses for a composite number.
\newblock {\em Mathematics of Computation}, 46(173):259--279, 1986.

\bibitem[FS20]{FioriShallue}
Andrew Fiori and Andrew Shallue.
\newblock Average liar count for degree-2 {F}robenius pseudoprimes.
\newblock {\em Mathematics of Computation}, 89(321):493--514, 2020.

\bibitem[FT91]{fouvry1991entiers}
{\'E}tienne Fouvry and G{\'e}rald Tenenbaum.
\newblock Entiers sans grand facteur premier en progressions arithm{\'e}tiques.
\newblock {\em Proceedings of the London Mathematical Society}, 3(3):449--494, 1991.

\bibitem[Ghe24]{Hiva}
Hiva Gheisari.
\newblock Studying the efficiency of the {F}robenius primality test.
\newblock Master's thesis, University of Lethbridge, 2024.

\bibitem[Gra01]{Grantham}
Jon Grantham.
\newblock {F}robenius pseudoprimes.
\newblock {\em Mathematics of Computation}, 70(234):873--891, 2001.

\bibitem[Gra10]{GRANTHAM20101117}
Jon Grantham.
\newblock There are infinitely many perrin pseudoprimes.
\newblock {\em Journal of Number Theory}, 130(5):1117--1128, 2010.

\bibitem[Gra20]{Grantham2}
Jon Grantham.
\newblock An unconditional improvement to the running time of the quadratic frobenius test.
\newblock {\em Journal of Number Theory}, 210:476--480, 2020.

\bibitem[HW79]{Hardy}
Godfrey~Harold Hardy and Edward~Maitland Wright.
\newblock {\em An introduction to the theory of numbers}.
\newblock Oxford university press, 1979.

\bibitem[HW23]{helmreich2023tabulatingabsolutelucaspseudoprimes}
Chloe Helmreich and Jonathan Webster.
\newblock Tabulating absolute lucas pseudoprimes.
\newblock {\em arXiv 2306.17691}, 2023.

\bibitem[LEL23]{Legnongo2023BadWF}
Johnathan~Djella Legnongo, Tony Ezome, and Florian Luca.
\newblock Bad witnesses for a composite number.
\newblock {\em Acta Arithmetica}, 2023.

\bibitem[Lic22]{lichtman2022primesarithmeticprogressionslarge}
Jared~Duker Lichtman.
\newblock Primes in arithmetic progressions to large moduli, and shifted primes without large prime factors.
\newblock {\em arXiv 2211.09641}, 2022.

\bibitem[Lin44]{Linnik}
UV~Linnik.
\newblock On the least prime in an arithmetic progression {II}. {The Deuring--Heilbronn phenomenon}.
\newblock {\em Rec. Math.(Sbornik)}, 15(3):347--368, 1944.

\bibitem[Mon80]{Monier}
L.~Monier.
\newblock Evaluation and comparison of two efficient probabilistic primality testing.
\newblock {\em algorithms, Theoret. Comput. Sci.}, 12(1):97--–108, 1980.

\bibitem[PJ21]{pomerance2021thoughtspseudoprimes}
Carl Pomerance and Samuel S.~Wagstaff Jr.
\newblock Some thoughts on pseudoprimes.
\newblock {\em arXiv 2103.00679}, 2021.

\bibitem[Ste20]{stephan2020millionsperrinpseudoprimesincluding}
Holger Stephan.
\newblock Millions of perrin pseudoprimes including a few giants.
\newblock {\em arXiv 2002.03756}, 2020.

\bibitem[SW19]{Shallue2018FastTO}
Andrew Shallue and Jonathan Webster.
\newblock Fast tabulation of challenge pseudoprimes.
\newblock {\em The Open Book Series}, 2(1):411--423, 2019.

\bibitem[Xyl11]{Xylouris}
Triantafyllos Xylouris.
\newblock On the least prime in an arithmetic progression and estimates for the zeros of {Dirichlet {L}-functions}.
\newblock {\em Acta Arithmetica}, 150:65--91, 2011.

\end{thebibliography}

\end{document}